\documentclass[12pt]{amsart}

\usepackage{amsmath,amssymb,amscd}
\usepackage[mathscr]{eucal}
\usepackage{setspace}
\usepackage{multirow}
\usepackage{array}
\usepackage{tabularx}
\usepackage{tikz}
\usetikzlibrary{matrix,arrows,decorations.pathmorphing}

\newtheorem{prop}{Proposition}[section]
\newtheorem{lem}[prop]{Lemma}
\newtheorem{defn}[prop]{Definition}
\newtheorem{thm}[prop]{Theorem}
\newtheorem{cor}[prop]{Corollary}

\newtheorem*{q1}{Question I}
\newtheorem*{q2}{Question II}
\newtheorem*{thm1}{Theorem 1.13: q = 1}
\newtheorem*{thm2}{Theorem 1.16: q = 2}
\newtheorem*{thm3}{Theorem 1.17: q = 3}

\newtheorem*{ex5.5}{Example 5.5: n = 2 or 6}
\newtheorem*{tab118}{Table 1.18}

\newcolumntype{Z}{>{\centering\arraybackslash}X}

\numberwithin{equation}{section}

\begin{document}
\title{Kervaire invariants and selfcoincidences}
\author{Ulrich Koschorke} 
\address{Universit\"{a}t Siegen\\
Emmy Noether Campus, Walter-Flex-Str. 3\\
D-57068 Siegen, Germany}
\email{koschorke@mathematik.uni-siegen.de}
\urladdr{http://www.math.uni-siegen.de/topologie/}
\author{Duane Randall}
\address{Loyola University New Orleans\\
Mathematical Sciences\\
6363 St. Charles Ave.\\
Campus Box 35\\
New Orleans, LA 70118}
\email{randall@loyno.edu}
\thanks{Both authors gratefully acknowledge the support by the "Research in Pairs"-program of the Mathematisches Forschungsinstitut Oberwolfach as well as by DFG (Deutsche Forschungsgemeinschaft) and by a Distinguished Professorship at Loyola University.}
\subjclass[2010]{Primary: 55 M 20, 55 P 40, 55 Q 15, 55 Q 25, 57 R 99\\
Secondary: 55 Q 45}

\begin{abstract}
Minimum numbers decide e.g. whether a given map $f: S^m\rightarrow S^n/G$ from a sphere into a spherical space form can be deformed to a map $f'$ such that $f(x)\neq f'(x)$ for all $x\in S^m$.  In this paper we compare minimum numbers to (geometrically defined) Nielsen numbers (which are more computable).  In the stable dimension range these numbers coincide.  But already in the first nonstable range (when $m=2n-2$) the Kervaire invariant appears as a decisive additional obstruction which detects interesting geometric coincidence phenomena.  Similar results (involving e.g. Hopf invariants, taken mod 4) are obtained in the next seven dimension ranges (when $1<m-2n+3\leqslant 8$).  The selfcoincidence context yields also a precise geometric criterion for the open question whether the Kervaire invariant vanishes on the 126-stem or not.
\end{abstract}

\maketitle

\section{Introduction and statement of results}

Given two (continuous) maps $f_1, f_2: M\rightarrow N$ between polyhedra, we are interested in somehow measuring the size of the coincidence subspace
\begin{equation}
C(f_1,f_2):= \{x\in M\ |\  f_1(x)=f_2(x)\}
\end{equation}
of $M$, and in capturing mainly those 'essential' features which remain preserved by homotopies $f_i \sim f'_i$, $i=1,2$.  Thus the central problem of topological coincidence theory is to determine the minimum numbers
\begin{equation}
MC(f_1,f_2):=\mathrm{min}\{\#C(f'_1,f'_2)\ |\ f'_1\sim f_1, f'_2\sim f_2\}
\end{equation}
and
\begin{equation}
MCC(f_1,f_2):=\mathrm{min}\{\#\pi_0(C(f'_1,f'_2))\ |\ f'_1\sim f_1, f'_2\sim f_2\}
\end{equation}
of coincidence points and of coincidence pathcomponents, resp., within given pairs of homotopy classes (\textit{cf.} [Ko5]; compare also with similar notions in topological fixed point theory, e.g. in [Br], p. 9).

The special case where these minimum numbers vanish is of particular interest.
\setcounter{prop}{3}
\begin{defn}
We call the pair $(f_1,f_2)$ of maps $\mathbf{loose}$ if there are homotopies $f_1\sim f'_1$, $f_2\sim f'_2$ such that $f'_1(x)\neq f'_2(x)$ for all $x\in M$ (i.e. $f_1, f_2$ can be 'deformed away' from one another).
\end{defn}\vspace{-4mm}

When $M$ and $N$ are smooth manifolds and $M$ is closed (and hence $MCC(f_1,f_2)$ - but not necessarily $MC(f_1,f_2)$ - is finite) there is a strong lower bound for both minimum numbers at our disposal, namely the Nielsen number $N^{\#}(f_1,f_2)$ (\textit{cf.} [Ko3] or [Ko5]).  It is a nonnegative integer which depends only on the homotopy classes of $f_1$ and $f_2$ and can be extracted from a careful geometric analysis of just one generic pair of representing maps.

\begin{q1}
When is $MCC(f_1,f_2)$ $\mathrm{different}$ from its lower bound $N^{\#}(f_1,f_2)$?
\end{q1}\vspace{-4mm}

In this paper we study the very special case $M=S^m$, $N=$ spherical space form (\textit{i.e.} $N=S^n/G$ where the finite group $G$ acts smoothly and freely on the sphere $S^n$).

\begin{thm}
Consider two maps $f_1,f_2: S^m\rightarrow S^n/G$, $m,n\geqslant 1$.\\  If $MCC(f_1,f_2)\neq N^{\#}(f_1,f_2)$, then $f_1\sim f_2$ (\textit{i.e.} $f_1,f_2$ are homotopic).
\end{thm}\vspace{-4mm}

Since the minimum and Nielsen numbers depend only on the homotopy classes, it suffices to consider the selfcoincidence case $f_1=f_2=:f$ in order to get interesting answers to our Question I.  (We may even assume that $f$ preserves base points.)

As a special feature of the selfcoincidence setting we have the following refinement of Definition 1.4 (introduced by Dold and Gon\c{c}alves; \textit{cf.} [DG], p. 296).

\begin{defn}
Given a map $f: M\rightarrow N$ between manifolds, put $f_1=f_2=:f$.  The pair $(f_1,f_2)=(f,f)$ is $\mathbf{loose\ by\ small\ deformation}$ if and only if for every metric on $N$ and every $\epsilon >0$ there exists an $\epsilon$-approximation $f'$ of $f$ such that $f'(x)\neq f(x)$ for all $x\in S^m$.
\end{defn}

\begin{q2}
When is $(f,f)$ loose, but not by small deformation?
\end{q2}\vspace{-4mm}

In our special case $M=S^m$, $N=S^n/G$ (compare 1.5) Question II turns out to be very closely related to our Question I above.  In order to see that, consider the (horizontal) exact homotopy sequence
$$
\begin{tikzpicture}
\matrix(m)[matrix of math nodes,
row sep=.6em, column sep=2.4em,
text height=1.5ex, text depth=0.25ex]
{& \cdots & \pi_m(V_{n+1,2}) & \pi_m(S^n) & \pi_{m-1}(S^{n-1}) & \pi_{m-1}(V_{n+1,2}) & \cdots\\
(1.7)\hspace{-.35in} & & & & & & \\
 & & & & \pi_m(S^n) & & \\};
\path[->,font=\scriptsize,>=angle 90]
(m-1-2) edge (m-1-3)
(m-1-3) edge (m-1-4)
(m-1-4) edge node[auto] {$\partial$} (m-1-5)
(m-1-5) edge node[auto] {$\mathrm{incl}_*$} (m-1-6)
(m-1-6) edge (m-1-7);
\path[dashed,->,font=\scriptsize,>=angle 90]
(m-1-5) edge node[auto] {$E:= \mathrm{suspension}$} (m-3-5);
\end{tikzpicture}
$$
of the Stiefel manifold $V_{n+1,2} = ST(S^n)$ of unit tangent vectors of $S^n$, fibered over $S^n$.
\setcounter{prop}{7}
\begin{defn}
We call the assumption
$$
0=\partial (\pi_m(S^n))\cap \mathrm{ker}\left(E: \pi_{m-1}(S^{n-1})\rightarrow \pi_m(S^n)\right)
$$
the $\mathbf{Wecken\ condition\ for\ (m,n)}$ (briefly: $\mathbf{WeC(m,n)}$).
\end{defn}\vspace{-4mm}

This condition holds at least if $n$ is odd (and hence $\partial (\pi_m(S^n))=0$ due to the existence of a nowhere zero vector field on $S^n$) or in the 'stable dimension range' $m<2n-2$ (where $\mathrm{ker}E=0$) or if $m\leqslant n+5$, $(m,n)\neq (11,6)$ or if $n=2\leqslant m$.

The previous discussion motivates us to study the five geometric or homotopy theoretical conditions $(\mathrm{i}),\ldots (\mathrm{v})$ written down below.  They are related to one another and to the Wecken condition 1.8 by the following result.  (For proofs and further details see Section 2 below and e.g. [Ko5]).

\begin{thm}
Given a finite group $G$ acting smoothly and freely on $S^n$, let $N=S^n/G$ be the resulting orbit space and consider any $[f]\in \pi_m(N)$, $m,n\geqslant 1$.  Then
$$
MC(f,f) = MCC(f,f)
$$
and these minimum numbers (as well as the Nielsen number $N^{\#}(f,f)$) can take only $0$ and $1$ as possible values.\\
\indent Moreoever, if $[f]$ lifts to some homotopy class $[\widetilde{f}]\in \pi_m(S^n)$ we have the following logical implications:
\begin{align*}
&(\mathrm{i})\ \partial ([\widetilde{f}])\in\pi_{m-1}(S^{n-1})\ vanishes;\\
&\Updownarrow \\
&(\mathrm{ii})\ (f,f)\ is\ loose\ by\ small\ deformation;\\
&\Downarrow\ \ (\Updownarrow\ e.g.\ if\ G\neq 0)\\
&(\mathrm{iii})\ MCC(f,f)=0;\ equivalently,\ (f,f)\ is\ loose\ (by\ any\ deformation);\hspace{.74in}\\
&\Downarrow\ \ (\Updownarrow\ e.g.\ if\ G\ncong\mathbb{Z}_2)\\
&(\mathrm{iv})\ N^{\#}(f,f)=0;\\
&\Updownarrow \\
&(\mathrm{v})\ E\circ\partial([\widetilde{f}])=0.
\end{align*}
In particular, the five conditions (i) - (v) are equivalent for all maps $f: S^m\rightarrow N$ if and only if the Wecken condition $WeC(m,n)$ holds (cf. 1.8) (and in this case $MC(f,f)=MCC(f,f) = N^{\#}(f,f)$ for all $f$).
\end{thm}

In view of this theorem we may say that $N^{\#}(f,f)$ is '\textit{at most one desuspension short}' of being a complete looseness obstruction.  Note also that
$$
MC(\widetilde{f},\widetilde{f}) = MCC(\widetilde{f},\widetilde{f}) = N^{\#}(\widetilde{f},\widetilde{f}) = N^{\#}(f,f).
$$
\indent The following equivalent version of conditions (iv) and (v) above is sometimes more suitable for computations.

\begin{cor}
Let $n\geqslant 2$ be even.  In the setting of Theorem 1.9 we have:  $N^{\#}(f,f)$ (or, equivalently, $N^{\#}(\widetilde{f},\widetilde{f})$) vanishes if and only if
\begin{itemize}
\item[(vi)] $2[\widetilde{f}] = [\iota_n,\iota_n]\circ h_0(\widetilde{f}).$
\end{itemize}
Here $h_0: \pi_m(S^n)\rightarrow \pi_m(S^{2n-1})$ denotes the $0^{th}$ Hopf-Hilton homomorphism (\textit{cf.} [W], ch. XI, 8.3).  (If $m\leqslant 3n-3$, then condition (vi) takes also the form 
\begin{itemize}
\item[($\mathrm{vi}'$)] $2[\widetilde{f}] = [\iota_n, E^{-n+1}(h_0(\widetilde{f}))]$
\end{itemize}
where $ E^{-n+1}$ denotes the inverse of the iterated suspension isomorphism $E^{n-1}$.)\\
\indent Now assume that $E: \pi_{m-1}(S^{n-1})\rightarrow \pi_m(S^n)$ is surjective.  Then
$$
N^{\#}(f,f) = N^{\#}(\widetilde{f},\widetilde{f}) = \begin{cases}
0 & \text{if $2[\widetilde{f}]=0$};\\
1 & \text{otherwise.}
\end{cases}
$$
\indent Assume in addition that $\mathrm{ker}E\cong \mathbb{Z}_2$ with generator $v$.  Then the Wecken condition $WeC(m,n)$ fails if and only if
$$
v\in (\iota_{n-1} + \iota_{n-1})_*(\pi_{m-1}(S^{n-1})).
$$
If also $(\iota_{n-1} + \iota_{n-1})_*\equiv 2\cdot\mathrm{id}_{\pi_{m-1}(S^{n-1})}$ then $\partial([\widetilde{f}]) = 2\cdot E^{-1}([\widetilde{f}])$ is obtained by multiplying any desuspension of $[\widetilde{f}]$ with 2; moreover, $WeC(m,n)$ fails precisely when $v$ can be halved, i.e. $v=2u$ for some $u\in\pi_{m-1}(S^{n-1})$.
\end{cor}

When can failures of the Wecken condition occur, and which geometric consequences do they have?  It turns out that they are directly related to our Questions I and II.

\begin{cor}
Assume that $G$ is nontrivial.  Then, given $[f]\in\pi_m(S^n/G)$ and a lifting $[\widetilde{f}]\in\pi_m(S^n)$ of $[f]$, $m,n\geqslant 1$, the following conditions are equivalent:
\begin{itemize}
\item[(i)] $\partial([\widetilde{f}])\neq 0$ but $E\circ\partial([\widetilde{f}])=0$;

\item[(ii)] $MCC(f,f)\neq N^{\#}(f,f)$;

\item[(iii)] $N^{\#}(f,f)=0$ but $f$ is coincidence producing (i.e. the pair $(f,f')$ cannot be loose for $\mathrm{any}$ map $f': S^m\rightarrow N$; thus $MCC(f,f')\neq 0$);

\item[(iv)] $MCC(f,f) > MCC(\widetilde{f},\widetilde{f})$;

\item[(v)] $MC(f,f) > MC(\widetilde{f},\widetilde{f})$;

\item[(vi)] $(\widetilde{f},\widetilde{f})$ is loose, but $(f,f)$ is not loose;

\item[(vii)] $(\widetilde{f},\widetilde{f})$ is loose, but not by small deformation.
\end{itemize}
\end{cor}

None of these conditions can hold when $G\ncong\mathbb{Z}_2$ (since then $\chi (N)\cdot \# G\neq\chi (S^n)$) or when $n\not\equiv 0$ (2) or in the 'stable dimension range' $m-2n+3\leqslant 0$ where our suspension homomorphism
$$
E: \pi_{m-1}(S^{n-1})\rightarrow \pi_m(S^n)
$$
(\textit{cf.} 1.7 and 1.8) in injective by Freudenthal's theorem.

\begin{defn}
We call the integer
$$
q := m-2n+3
$$
the $\mathbf{degree\ of\ nonstability\ of\ E}$.
\end{defn}

In view of Theorem 1.9 and Corollary 1.11 let us take a closer look at the first few nonstable dimension settings $q=1,2,3,\ldots$ when $n$ is even and $G=\mathbb{Z}_2$ and hence $N=S^n/G$ is the orbit space of a fixed point free involution on $S^n$ (standard example: $N=\mathbb{RP}^n$).

Already in the first dimension setting outside of the stable range we encounter fascinating interrelations of coincidence theory with other, seemingly distant branches of topology.

\begin{thm1}
For all $[f]\in\pi_{2n-2}(S^n/\mathbb{Z}_2)$, as well as $[\widetilde{f}]\in\pi_{2n-2}(S^n)$ lifting $[f]$, $n$ even, $n\geqslant 2$, we have:
\begin{equation*}
N^{\#}(f,f) = \begin{cases}
0 & \text{if $2[\widetilde{f}]=0$};\\
1 & \text{otherwise.}
\end{cases}
\end{equation*}
\indent If $n\neq 2, 4, 8$, then $MCC(f,f)=0$ precisely when $N^{\#}(f,f)$ and the Kervaire invariant $KI([\widetilde{f}]):= KI(E^{\infty}([\widetilde{f}]))$ of $[\widetilde{f}]$ vanish.
If $n=2, 4, 8$ then $MCC(f,f) = N^{\#}(f,f)$.\\
\indent The Wecken condition $WeC(2n-2,n)$ fails if and only if $n=16,32$ or $64$ (or possibly $128$); in these dimensions $[f]$, $[\widetilde{f}]$ have the seven equivalent properties $(\mathrm{i}),\ldots,(\mathrm{vii})$ in Corollary 1.11 precisely if $[\widetilde{f}]$ has order $2$ and Kervaire invariant one.
\end{thm1}

Here we define the composite homomorphism $KI = KI\circ E^{\infty}$,
\begin{flushleft}
\begin{tikzpicture}
\matrix(m)[matrix of math nodes,
row sep=2.6em, column sep=2.8em,
text height=1.5ex, text depth=0.25ex]
{(1.14)\hspace{.82in} & KI: \pi_{2n-2}(S^n) & \pi_{n-2}^s & \mathbb{Z}_2,\\};
\path[->,font=\scriptsize,>=angle 90]
(m-1-2) edge node[above] {$E^{\infty}$} node[below] {$\cong$} (m-1-3)
(m-1-3) edge node[above] {$KI$} (m-1-4);
\end{tikzpicture}
\end{flushleft}
by the Kervaire invariant when $n\equiv 0$ (4) and by the trivial homomorphism otherwise; recall that the stable stem $\pi^s_{n-2}$ can also be interpreted as the $(n-2)$th framed bordism group.

Originally M. Kervaire introduced his invariant in order to exhibit a triangulable closed manifold which does not admit any differentiable structure (\textit{cf.} [K]).  Subsequently M. Kervaire and J. Milnor used it in their classification of exotic spheres (\textit{cf.} [KM]; see also [M]).  Then W. Browder showed that, given $[\widetilde{f}]\in\pi_{2n-2}(S^n)$, $KI([\widetilde{f}])=0$ whenever $n$ is not a power of 2 (\textit{cf.} [B]).  But for $n=16, 32$ or 64 there exists an element $[\widetilde{f}]\in\pi_{2n-2}(S^n)$ such that $KI([\widetilde{f}])=1$ and $2[\widetilde{f}]=0$ ('strong Kervaire invariant one conjecture'; \textit{cf.} [BJM1], pp. 10-11, and [Co], p. 1195); e.g. the authors of [BJM2] construct Kervaire invariant one elements in $\pi^s_{62}\cong\mathbb{Z}_4\oplus \mathbb{Z}_2\oplus \mathbb{Z}_2\oplus \mathbb{Z}_3$ (\textit{cf.} [KMa]).  On the other hand, according to the spectacular recent results of M. Hill, M. Hopkins and D. Ravenel $KI([\widetilde{f}])\equiv 0$ whenever $n>128$ (\textit{cf.} [HHR]).  Only the case $n=128$ remains open.  This open question turns out to be closely related to selfcoincidences.
\setcounter{prop}{14}
\begin{thm}
The following conditions are equivalent:
\begin{itemize}
\item[(i)] there exists no Kervaire invariant one element in $\pi^s_{126}\cong \pi_{254}(S^{128})$ (in particular, $KI(E^{12}([h]))=0$ for all $[h]\in\pi_{242}(S^{116}))$;

\item[(ii)] there exists a suspended element $[Eg]$ in $E(\pi_{240}(S^{115}))\subset \pi_{241}(S^{116})$ such that $(Eg,Eg)$ is loose, but not loose by small deformation.
\end{itemize}
\end{thm}

The results 1.13 and 1.15 involving the Kervaire invariant will be proved in Section 4 below.

In the next nonstable dimension setting the Wecken condition fails in infinitely many exceptional dimension combinations.  Recall that for $n$ even the group $\pi_{2n-1}(S^n)$ is the direct sum of its torsion subgroup and a copy of $\mathbb{Z}$ generated by the Hopf maps when $n=2,4$ or 8, and by $[\iota_n,\iota_n]$ otherwise.

\begin{thm2}
Consider any $[f]\in\pi_{2n-1}(S^n/\mathbb{Z}_2)$, and $[\widetilde{f}]\in\pi_{2n-1}(S^n)$ lifting $[f]$.\\
\indent If $n\neq 4,8$, $n$ even, then
\begin{equation*}
N^{\#}(f,f) = \begin{cases}
0 & \text{if the torsion part of $[\widetilde{f}]$ has order $\leqslant 2$};\\
1 & \text{otherwise}.
\end{cases}
\end{equation*}
\indent If $n\equiv 0$ $(4)$ or $n=2$, then $MCC(f,f)=N^{\#}(f,f)$.\\
\indent If $n\equiv 2$ $(4)$, $n\geqslant 6$, the Wecken condition $WeC(2n-1,n)$ does \underline{not} hold; moreover, $MCC(f,f)=0$ precisely when $N^{\#}(f,f)=0$ and the Hopf invariant $H(\widetilde{f})$ of $\widetilde{f}$ is divisible by $4$; in particular, $[f]$ and $[\widetilde{f}]$ satisfy the seven equivalent conditions $(\mathrm{i}),\ldots,(\mathrm{vii})$ in 1.11 precisely when the torsion part of $[\widetilde{f}]$ has order $\leqslant 2$ and $H(\widetilde{f})\equiv 2$ $(4)$.
\end{thm2}

This and the following theorem will be proved in Section 5 below.

\begin{thm3}
Consider any $[f]\in \pi_{2n}(S^n/\mathbb{Z}_2)$, and $[\widetilde{f}]\in\pi_{2n}(S^n)$ lifting $[f]$.\\
\indent If $n\equiv 0$ $(4)$ or $n=2$ or $n=6$, then $\mathrm{ker}E=0$ and
\begin{equation*}
MCC(f,f) = N^{\#}(f,f) = \begin{cases}
0 & \text{if $2[\widetilde{f}]=[\iota_n,\iota_n]\circ h_0(\widetilde{f})$};\\
1 & \text{otherwise}.
\end{cases}
\end{equation*}
(Here $h_0$ denotes the $0$th Hopf-Hilton invariant as in 1.10, $(\mathrm{vi})$).  \\
\indent Now assume that $n\equiv 2$ $(4)$ and $n\geqslant 10$.  Then $E$ is surjective with kernel $\mathrm{ker}E\cong \mathbb{Z}_2$, generated by $[\iota_{n-1}, \eta^2_{n-1}]\in 4\cdot\pi_{2n-1}(S^{n-1})$; in particular a desuspension $E^{-1}([\widetilde{f}])\in \pi_{2n-1}(S^{n-1})$ of $[\widetilde{f}]$ exists, and $2\cdot E^{-1}[\widetilde{f}])$ is well-defined.  Moreover the Wecken condition $WeC(2n,n)$ fails to hold and we have
\begin{equation*}
N^{\#}(f,f) = \begin{cases}
0 & \text{if $2[\widetilde{f}]=0$};\\
1 & \text{otherwise};
\end{cases}
\end{equation*}
and
\begin{equation*}
MCC(f,f) = \begin{cases}
0 & \text{if $2\cdot E^{-1}([\widetilde{f}])=0$};\\
1 & \text{otherwise}.
\end{cases}
\end{equation*}
Thus $MCC(f,f)\neq N^{\#}(f,f)$ if and only if\ \ $2[\widetilde{f}]=0$ but $2\cdot E^{-1}([\widetilde{f}])=[\iota_{n-1},\eta^2_{n-1}]$ (compare 1.10).
\end{thm3}

Further results concerning the Wecken condition in the first 8 nonstable dimension settings can be found in \textbf{Table 1.18} below.  They are based on complete injectivity and surjectivity criteria for $E$ (extracted from the literature and listed in the first two columns of \textbf{Table 1.18}) as well as from the following proposition.

\setcounter{prop}{18}
\begin{prop}
Assume $q\leqslant 8$ (\textit{cf.} 1.12) and $n\equiv 0$ $(2)$, $n\geqslant 2$.  Then either $\mathrm{ker}E=0$ (and therefore the Wecken condition holds) or $\mathrm{ker} E\cong\mathbb{Z}_2$ (and nearly always generated by a suitable Whitehead product).
\end{prop}

This follows from the work of many authors, collected and completed by M. Golasi\'{n}ski and J. Mukai in [GM].  Actually these sources allow us also to decide which of the two alternatives hold.

\begin{flushleft}
\textbf{Example 1.20: q=6.}  If $n\geqslant 2$ is even and a discrete group $G$ acts freely and smoothly on $S^n$, then we have for every map $f: S^{2n+3}\rightarrow S^n/G$
\begin{equation*}
MC(f,f) = MCC(f,f) = N^{\#}(f,f) = \begin{cases}
0 & \text{if $2[f]=0$ (e.g. if $n=2, 6$ or 14)};\\
1 & \text{otherwise}.
\end{cases}
\end{equation*}
This is mainly due to the fact that the stable homotopy groups $\pi^s_4$ and $\pi^s_5$ vanish.
\end{flushleft}
\begin{tab118}
The suspension homomorphism $E: \pi_{m-1}(S^{n-1})\rightarrow \pi_m(S^n)$ and the Wecken condition for $m,n\geqslant 2$, $n$ even.  (We adopt the notation of [GM]; e.g. $\eta_2$, $\nu_4$, $\sigma_8$ denote the Hopf maps).
\end{tab118}
\begin{center}
\begin{tabularx}{6.4in}{|Z|Z|Z|Z|}
\hline
$m$ & $E$ is injective iff & $E$ is onto iff & The Wecken condition $WeC(m,n)$ fails to hold\\\hline

$m\leqslant 2n-3$\hspace{1in} $(q\leqslant 0)$ & always & always & never\\\hline

$m=2n-2$\hspace{1in} $(q=1)$ & $n=2,4,8$ & always & iff $n=16$, 32 or 64 (or possibly 128)\\\hline

$m=2n-1$\hspace{1in} $(q=2)$ & $n\equiv 0$ (4) or $n=2$ & never & iff $n\equiv 2$ (4), $n\geqslant 6$ \\\hline

$m=2n$\hspace{1in} $(q=3)$ & $n\equiv 0$ (4) or $n=2,6$ & $n\equiv 2$ (4), $n\geqslant 6$ & iff $n\equiv 2$ (4), $n\geqslant 10$ \\\hline

$m=2n+1$\hspace{1in} $(q=4)$ & $n\equiv 0$ (8) or $n=4$ or $n=2^i-2$ for $i\geqslant 2$ & $n\equiv 2$ (4), $n\geqslant 10$ & e.g. if $n\equiv 2$ $(4)$ and $n\geqslant 10$ and $n\neq 2^i-2$ for $i\geqslant 4$ and $[\iota_{n-1},\nu_{n-1}]$ can be halved \\\hline

$m=2n+2$\hspace{1in} $(q=5)$ & always & never & never \\\hline

$m=2n+3$\hspace{1in} $(q=6)$ & always & $n\geqslant 4$ & never \\\hline

$m=2n+4$\hspace{1in} $(q=7)$ & $n\equiv 6,8$ (8) or $n=2^i-4$ for $i\geqslant 3$ or $n=2$ & $n\geqslant 4$ &  iff $n\equiv 2,4$ (8) and $n\geqslant 10$ and $n\neq 2^i-4$ for $i\geqslant 4$ and $[\iota_{n-1},\nu^2_{n-1}]$ can be halved\\\hline

$m=2n+5$\hspace{1in} $(q=8)$ & $n\equiv 0$ (16)  or $n=2,4,8,12$ & $n\equiv 2,4$ (8) and $n\neq 2^i-4$ for $i\geqslant 4$ and $n\geqslant 10$ & e.g. if $n\not\equiv 0$ (16) and $n\geqslant 14$ and $[\iota_{n-1},\sigma_{n-1}]$ can be halved\\\hline
\end{tabularx}
\end{center}

\section{Minimum numbers and Nielsen numbers}

Minimum numbers and four types of Nielsen numbers (agreeing with the classical notions in the setting of fixed point theory) were discussed in great detail in the survey paper [Ko5], and so were various 'Wecken theorems' specifying conditions where a minimum number equals a Nielsen number.  In particular, proofs of the results 1.5, 1.9 and 1.11 were indicated in [Ko5], sometimes in the more general context of maps from $S^m$ into an arbitrary smooth connected $n$-dimensional manifold $N$ (for more background see [Ko1] - [Ko4]; see also [C] for some work on related questions).

It is not surprising that the boundary homomorphism
$$
\partial_N: \pi_m(N)\rightarrow \pi_{m-1}(S^{n-1})
$$
of the homotopy exact sequence of the tangent sphere bundle $ST(N)$, fibered over $N$, plays an important role.  So let us describe it geometrically.

Decompose $S^m$ into half-spheres $S^m_+$ and $S^m_-$, and consider the pullback $\eta = f^*(ST(N))$ of the tangent sphere bundle of $N$ by $f: S^m\rightarrow N$.  Choose a section $s$ of $\eta|_{S^m_-}$ and use a (suitably orientation preserving) trivialization $\eta|_{S^m_+}\cong S^m_+\times S^{n-1}$ to interpret $s|_{S^{m-1}}$ as an 'index map' into the fibre $S^{n-1}$.  The resulting homotopy class agrees with $\partial_N[f]$.

Often it seems difficult to compute $\partial_N$, even when $N=S^n$ (and hence $\partial_N=\partial$, \textit{cf.} 1.7).  However, here is a well known partial result.

\begin{lem}
The composite homomorphism
$$
\partial\circ E: \pi_{m-1}(S^{n-1})\rightarrow \pi_{m-1}(S^{n-1})
$$
is induced by the self map of $S^{n-1}$ having degree $\chi (S^n)=1+(-1)^n$.
\end{lem}\vspace{-4mm}

\begin{proof}
If $f$ is the suspension of some map $f': S^{m-1}\rightarrow S^{n-1}$ in the geometric description of $\partial_N(f)$ above, the decompositions of $S^m$ and $S^n$ into the half-spheres are preserved.  Thus $s$ can be chosen to be the pullback of a vector field on $S^n$ with index $\chi (S^n)$.
\end{proof}\vspace{-4mm}

This gives no new insights when $n$ is odd and hence $\partial\equiv 0$ on the whole group $\pi_m(S^n)$ (and not just on $E(\pi_{m-1}(S^{n-1}))$), due to a nowhere vanishing vector field on $S^n$ and the resulting splitting of 1.7.

If $n$ is even and $\alpha\in\pi_{m-1}(S^{n-1})$ then according to Theorem 8.9 in [W], ch. XI, we have
\setcounter{equation}{1}
\begin{equation}
\partial\circ E(\alpha) = (2\cdot\iota_{n-1})\circ\alpha = 2\alpha + [\iota_{n-1},\iota_{n-1}]\circ h_0(\alpha) - \big[[\iota_{n-1},\iota_{n-1}],\iota_{n-1}\big]\circ h_1(\alpha)
\end{equation}
where $h_0$ and $h_1$ denote Hopf-Hilton homomorphisms.  (The last term to the right vanishes trivially when $m\leqslant 3n-5$, \textit{i.e.} $n\geqslant q+2$, since then $h_1(\alpha)\in\pi_{m-1}(S^{3n-5})=0$.)

\begin{flushleft}\textit{Proof of Corollary 1.10.}  We apply Theorem 1.9 to the case when $G$ is trivial and conclude that condition (v) is equivalent to $(\widetilde{f},\widetilde{f})$ being loose.  But this means that $\widetilde{f}$ is homotopic to its composite with the antipodal map on $S^n$ or, equivalently, that $[\widetilde{f}] = (-\iota_n)\circ [\widetilde{f}]$ (\textit{cf.} [DG], 2.10).  Thus Corollary 1.10 follows from [W], ch. XI, 8.12 and ch. XII, 2.4-2.5, as well as from Lemma 2.1 above.\hspace{4.51in}$\Box$\end{flushleft}
\section{Whitehead products and suspensions}

Since we are looking for interesting answers to our Questions I and II \textit{we assume from now on that $n$ is even}.  Let us first discuss the claims in Table 1.18 concerning the kernel and cokernel of $E: \pi_{m-1}(S^{n-1})\rightarrow \pi_m(S^n)$.  They are based on work of Adams, Barcus, Barratt, Freudenthal, Golasinski, Hilton, Hoo, James, Kristensen, Madsen, Mahowald, Mimura, Mori, Mukai, Oda, Oshima, Serre, Thomeier, Toda, Whitehead and others, as summarized in [GM] and quoted below.

Using the notation of [GM] we write $\eta_2\in\pi_3(S^2)$, $\nu_4\in\pi_7(S^4)$ and $\sigma_8\in\pi_{15}(S^8)$ for the Hopf maps.  Suspend and/or compose them to obtain the elements
\begin{align}
\iota_j & =[\mathrm{identity\ map}]\in\pi_j (S^j),\ j\geqslant 1;\notag\\
\eta_j & \in\pi_{j+1}(S^j)\ \mathrm{and}\ \eta^2_j = \eta_j\circ\eta_{j+1}\in\pi_{j+2}(S^j),\ j\geqslant 2;\\
\nu_j & \in\pi_{j+3}(S^j)\ \mathrm{and}\ \nu^2_j = \nu_j\circ\nu_{j+3}\in\pi_{j+6}(S^j),\ j\geqslant 4;\notag\\
\sigma_j & \in\pi_{j+7}(S^j),\ j\geqslant 8\notag.
\end{align}

According to the Hopf-invariant-one result of Adams [A] and to Section 2 of [GM] we have:
\begin{align*}
&\left.\hspace{2pt}(3.2) \ \ \# [\iota_j,\iota_j]  = \begin{cases}
1 & \text{if $j= 1, 3$ or 7};\\
2 & \text{if $j$ is odd, $j\neq 1, 3, 7$};\\
\infty & \text{if $j$ is even}.
\end{cases}\hspace{2.47in}\right\} (j\geqslant 1)\\
&\left. \begin{aligned}
(3.3)\ \ \ \ \hspace{1pt}\mathrm{[}\iota_j,\eta_j\mathrm{]} & = 0\ \ \mathrm{if\ and\ only\ if}\ j\equiv 3\ (4)\ \mathrm{or}\ j=2,6\\
(3.4)\ \ \ \ [\iota_j,\eta^2_j] & = 0\ \ \mathrm{if\ and\ only\ if}\ j\equiv 2, 3\ (4)\ \mathrm{or}\ j=5. 
\end{aligned}\hspace{1.83in}\right\} (j\geqslant 2)\\
&\left. \begin{aligned}
(3.5)\ \  \# \mathrm{[}\iota_j,\nu_j\mathrm{]} & = \begin{cases}
1 & \text{if $j\equiv 7$ $(8)$ or $j=2^i-3$ for $i\geqslant 3$};\\
2 & \text{if $j\equiv 1,3,5$ $(8)\geqslant 9$ and $j\neq 2^i-3$};\\
12 & \text{if $j\equiv 2$ $(4)\geqslant 6$ or $j=4,12$;}\\
24 & \text{if $j\equiv 0$ $(4)\geqslant 8$ unless $j=12$}.
\end{cases}\\
(3.6)\ \ \ \  [\iota_j,\nu^2_j] & =\ 0\ \ \mathrm{if\ and\ only\ if}\ j\equiv 4, 5, 7\ (8)\ \mathrm{or}\ j=2^i-5\ \mathrm{for}\ i\geqslant 4.
\end{aligned}\hspace{.61in}\right\} (j\geqslant 4)\\
&\left.\hspace{2pt}
(3.7) \ \ \# [\iota_j,\sigma_j]  = \begin{cases}
1 & \text{if $j=11$ or $j\equiv 15$ $(16)$};\\
2 & \text{if $j$ is odd and $j\geqslant 9$ unless $j=11$ or $j\equiv 15$ $(16)$;}\\
120 & \text{if $j=8$};\\
240 & \text{if $j$ is even and $j\geqslant 10$}.
\end{cases}\qquad\right\} (j\geqslant 8)
\end{align*}
Here $\#$ denotes the order of a group element.

Now assume that $m\leqslant 3n-5$ and consider the exact sequence
\begin{flushleft}
\begin{tikzpicture}
\matrix(m)[matrix of math nodes,
row sep=2.6em, column sep=3.0em,
text height=1.5ex, text depth=0.25ex]
{(3.8)\hspace{-.4in} & \pi_{m-n+1}(S^{n-1}) & \pi_{m-1}(S^{n-1}) & \pi_m(S^n) & \pi_{m-n}(S^{n-1}) & \cdots\\};
\path[->,font=\scriptsize,>=angle 90]
(m-1-2) edge node[auto] {$[\iota_{n-1}, - ]$} (m-1-3)
(m-1-3) edge node[auto] {$E$} (m-1-4)
(m-1-4) edge node[auto] {$E^{-n}\circ H$} (m-1-5)
(m-1-5) edge node[auto] {$[\iota_{n-1}, - ]$} (m-1-6);
\end{tikzpicture}
\end{flushleft}
derived from the EHP-sequence (\textit{cf.} [W] XII, 2.3 and 2.5).  The group $\pi_{m-n+1}(S^{n-1})$ is stable (\textit{i.e.} isomorphic to $\pi_{q-1}^s$, \textit{cf.} 1.12) and generated by one of the homotopy classes listed in (3.1) if $1\leqslant q\leqslant 8$, $q\neq 5,6$.  Thus the Whitehead product of $\iota_{n-1}$ with this generator vanishes precisely if the subsequent suspension homomorphism $E$ in (3.8) is injective.

Similarly $E$ is onto if and only if $E^{-n}\circ H\equiv 0$, \textit{i.e.} the generator of $\pi_{m-n}(S^{n-1})$ has the same order as its Whitehead product with $\iota_{n-1}$.  The claims in the first two columns of table 1.18 now follow from (3.1) - (3.7).

The remaining injectivity and surjectivity claims in the finitely many dimension combinations where $m> 3n-5$ and $q=m-2n+3\leqslant 8$ can be established by using Toda's explicit calculations in chapter V of [T] or the 2-primary sequence in [T], ch. IV, (4.4).  Indeed, since we assume $n$ to be even, $\mathrm{ker}E$ has no odd torsion (by Serre's theorem, \textit{cf.} [T], 13.1).\hspace{.37in}$\Box$

These techniques yield also a proof of Proposition 1.19.

\begin{flushleft}
\textbf{Example 3.9.}  If $q\leqslant 8$ and $n=4$, then $\mathrm{ker}E=0$.
\end{flushleft}

Clearly $E: \pi_{m-1}(S^{n-1})\rightarrow \pi_m(S^n)$ is also injective whenever $n=2\leqslant m$.

\section{Kervaire invariants}
In this section we discuss Theorems 1.13 and 1.15 which deal with the first nonstable dimension setting $m=2n-2$, $n$ even, $n\geqslant 2$.  The role which the Kervaire invariant plays for selfcoincidence questions was first noted in [GR1] and [GR2].
\setcounter{equation}{1}

First we prove Theorem 1.13, using Proposition 1.19.  We may assume that $n\neq 2,4,8$ (since $E$ is an isomorphism otherwise).  Then the EHP-sequence (\textit{cf.} 3.8) yields the short exact sequence
\begin{flushleft}
\begin{tikzpicture}
\matrix(m)[matrix of math nodes,
row sep=2.6em, column sep=2.8em,
text height=1.5ex, text depth=0.25ex]
{(4.1)\hspace{.28in} & 0 & \mathbb{Z}_2 & \pi_{2n-3}(S^{n-1}) & \pi_{2n-2}(S^n) & 0\\};
\path[->,font=\scriptsize,>=angle 90]
(m-1-2) edge (m-1-3)
(m-1-3) edge (m-1-4)
(m-1-4) edge node[auto] {$E$} (m-1-5)
(m-1-5) edge (m-1-6);
\path[dashed,->,font=\scriptsize,>=angle 90]
(m-1-5) edge [bend left=45] node[below] {$\partial$} (m-1-4);
\end{tikzpicture}
\end{flushleft}
where $\mathrm{ker}E\cong\mathbb{Z}_2$ is generated by
\begin{equation}
w_{n-1} = [\iota_{n-1},\iota_{n-1}].
\end{equation}
Moreover $\partial\circ E\equiv 2\cdot\mathrm{id}$ on $\pi_{m-1}(S^{n-1}) = E(\pi_{m-2}(S^{n-2})$ and hence $\partial\equiv 2\cdot E^{-1}$ (by 2.2).  Therefore the Wecken condition $WeC(2n-2,n)$ fails if and only if there is some $\alpha\in\pi_{2n-3}(S^{n-1})$ such that $2\alpha = w_{n-1}$, \textit{i.e.} $w_{n-1}$ can be halved.  But this is equivalent to the existence of a Kervaire invariant one element $[\widetilde{f}]\in \pi_{2n-2}(S^n)\cong \pi^s_{n-2}$ having order 2 (\textit{cf.} [BP]) which is known to exist e.g. for $n=16, 32$ and 64 (\textit{cf.} [BJM1], p. 11, and [Co], p. 1195).  In fact, for all order 2 Kervaire invariant one elements $[\widetilde{f}]$ we know that $\partial([\widetilde{f}]) = 2\cdot E^{-1}([\widetilde{f}]) = w_{n-1}$ (\textit{cf.} [BJM1], p. 11); hence the homomorphisms $\partial$ and $KI$ agree on all $[\widetilde{f}]\in\pi_{2n-2}(S^n)$ such that $2[\widetilde{f}]$ or equivalently, $N^{\#}(\widetilde{f},\widetilde{f})$ vanish.\hspace{3.36in}$\Box$

\begin{flushleft}\textit{Proof of Theorem 1.15.}  Toda [T] denotes the standard $\mathrm{im}J$ generator for the 2-primary component $\mathbb{Z}_8$ of the stable 11-stem $\mathbb{Z}_{504}$ by $\zeta$.  Now, Theorem 3.1 of [LR] affirms that the vanishing of the Whitehead product $[\zeta_{115},\iota_{115}]$ in $\pi_{240}(S^{115})$ is the complete obstruction to the existence of $\theta_6$.  If $[\zeta_{115},\iota_{115}]=0$, any map $h: S^{242}\rightarrow S^{116}$ with Hopf invariant $\zeta_{231}$ has the property that $KI([E^{12}h])=1$ by Theorem A of [BJM3].\end{flushleft}\vspace{-.12in}

We assume now that condition (i) holds, \textit{i.e.} $[\zeta_{115},\iota_{115}]\neq 0$, and shall derive (ii).  We shall deduce that $[\zeta_{115},\iota_{115}]\in 8\pi_{240}(S^{115})$ from Nomura's tables in [N, p. 164] (also see (3.2) of [LR]).  With respect to the Stiefel fibering $S^{115}\stackrel{i}{\rightarrow} V_{128,13}\rightarrow V_{128,12}$, Nomura shows that $i_*(\zeta_{115}) = 8i_8s_5\sigma_{119}\in \pi_{126}(V_{128,13})$.  Consider the $J$-morphism $J: SO(115)\rightarrow SF(115) = \Omega^{115}S^{115}$.  Since the boundary morphism $\partial: \pi_{126}(V_{128,13})\rightarrow \pi_{125}(SO(115))$ is an isomorphism in the fibration $SO(115)\rightarrow SO(128)\rightarrow V_{128,13}$, the image of $i_*(\zeta_{115})$ under the morphisms
$$
\begin{tikzpicture}
\matrix(m)[matrix of math nodes,
row sep=2.6em, column sep=2.8em,
text height=1.5ex, text depth=0.25ex]
{\pi_{126}(V_{128,13}) & \pi_{125}(SO(115)) & \pi_{125}(SF(115))\\};
\path[->,font=\scriptsize,>=angle 90]
(m-1-1) edge node[above] {$\partial$} node[below] {$\cong$} (m-1-2)
(m-1-2) edge (m-1-3);
\end{tikzpicture}
$$
is represented by
$$
\begin{tikzpicture}
\matrix(m)[matrix of math nodes,
row sep=2.6em, column sep=2.8em,
text height=1.5ex, text depth=0.25ex]
{S^{125} & S^{114} & SO(115) & SF(115).\\};
\path[->,font=\scriptsize,>=angle 90]
(m-1-1) edge node[above] {$\zeta_{114}$} (m-1-2)
(m-1-2) edge node[above] {$T(S^{115})$} (m-1-3)
(m-1-3) edge node[above] {$J$} (m-1-4);
\end{tikzpicture}
$$
The adjoint of the above map represents $[\zeta_{115},\iota_{115}]$.  Thus nontriviality of $[\zeta_{115},\iota_{115}]$ yields $[\zeta_{115},\iota_{115}]=8\alpha$ for some class $\alpha\in \pi_{240}(S^{115})$.  Let $g: S^{240}\rightarrow S^{115}$ represent $4\alpha$.  Then $[Eg]$ has order 2 so $(Eg,Eg)$ is loose.  Note that $h_0([g]) = h_0(4\alpha) = 4 h_0(\alpha) = 0$ in $\pi_{240}(S^{229})$ by Theorem 8.6 of [W], ch. XI.  Now in the Stiefel fibering $V_{117,2}\rightarrow S^{116}$, 
$$
\partial([Eg]) = (2\iota_{115})\circ [g]  = 2[g]+[\iota_{115},\iota_{115}] h_0([g]) = 2[g] = [\zeta_{115},\iota_{115}].
$$
Consequently, $(Eg,Eg)$ is not loose by small deformation. 

We now assume that condition (ii) holds and shall deduce that $[\zeta_{115},\iota_{115}]\neq 0$ so condition (i) must hold.  By hypothesis, for some map $g: S^{240}\rightarrow S^{115}$, $(Eg,Eg)$ is loose, but not loose by small deformation.  That is, $\partial ([Eg])\neq 0$ in $\pi_{240}(S^{115})$ in the Stiefel fibering $S^{115}\rightarrow V_{117,2}\rightarrow S^{116}$, but $E(\partial([Eg]))=0$ in $\pi_{241}(S^{116})$.  Now $(Eg,Eg)$ is loose means $(-\iota_{116})\circ [Eg] = -[Eg] = [Eg]$ so $2[Eg]=0$.  So $P: \pi_{242}(S^{231})\rightarrow \pi_{240}(S^{115})$ must be nontrivial on the 2-primary component of $\pi_{242}(S^{231})$ by exactness of the EHP-sequence.  Consequently, $P(\zeta_{231}) = [\zeta_{115},\iota_{115}]\neq 0$.\hspace{4.64in}$\Box$

\section{Criteria involving Hopf invariants}
In this section we prove Theorems 1.16 and 1.17 as well as the claims in the third column of Table 1.18, applying Corollary 1.10 and Proposition 1.19 which were established in sections 2 and 3.  Again let $n$ be even throughout.

Assume that $h_0: \pi_{m-1}(S^{n-1})\rightarrow \pi_{m-1}(S^{2n-3})$ vanishes and $\mathrm{ker}E$ is isomorphic to $\mathbb{Z}_2$.  If the generator $v$ of $\mathrm{ker}E$ can be halved, \textit{i.e.} $v=2\alpha$ for some $\alpha\in\pi_{m-1}(S^{n-1})$, then the Wecken condition $WeC(m,n)$ (\textit{cf.} 1.8) fails; indeed
$$
\partial(E(\alpha)) = 2\alpha + [\iota_{n-1},\iota_{n-1}]\circ h_0(\alpha) = v
$$
(\textit{cf.} 2.2) is a nontrivial element of $\partial(\pi_m(S^n))\cap \mathrm{ker}E$.  When $E$ is also onto, then $\partial(\pi_m(S^n)) = 2\cdot \pi_{m-1}(S^{n-1})$, and $WeC(m,n)$ fails precisely if $v$ can be halved.

In the dimension range $q\leqslant 8$ we know when $\mathrm{ker}E\cong\mathbb{Z}_2$ and when $E$ is onto (\textit{cf.} 1.18 and 1.19).  So let us take a closer look at $h_0$.

\begin{lem}
Assume that $q\leqslant 8$ and $\mathrm{ker}E\neq 0$.  Then $h_0(\pi_{m-1}(S^{n-1}))=0$ except when $q=4$ and $n\equiv 4$ $(8)$, $n\geqslant 12$, or possibly also when $q=8$ and $n=6$ or $10$.
\end{lem}

\begin{proof}
If $(q,n)\neq (8,6), (8,10)$ our assumptions imply that $n\geqslant q+3$ or, equivalently, $m-1\leqslant 3(n-1)-4$ (\textit{cf.} the first column of Table 1.18); thus the piece
$$
\begin{tikzpicture}
\matrix(m)[matrix of math nodes,
row sep=2.6em, column sep=3.2em,
text height=1.5ex, text depth=0.25ex]
{\pi_{m-2}(S^{n-2}) & \pi_{m-1}(S^{n-1}) & \pi_{m-n}(S^{n-2}) & \pi_{m-3}(S^{n-2})\\};
\path[->,font=\scriptsize,>=angle 90]
(m-1-1) edge node[above] {$E'$} (m-1-2)
(m-1-2) edge node[above] {$E'^{1-n}\circ H'$} (m-1-3)
(m-1-3) edge node[above] {$[\iota_{n-2}, - ]$} (m-1-4);
\end{tikzpicture}
$$
of the EHP-sequence (for $m-1, n-1$ instead of $m, n$ in 3.8) is exact and $h_0 = \pm H'$ (\textit{cf.} [BS], Theorem 4.18) vanishes precisely if the generator $\gamma_{n-2}$ of $\pi_{m-n}(S^{n-2})$ has the same order as its image $[\iota_{n-2},\gamma_{n-2}]$.  According to formulas (3.2) - (3.7) above this holds except when $q=4$, $n\equiv 4$ (8) and $n> 4$ (\textit{i.e.} except when $\# [\iota_{n-2},\nu_{n-2}]=12$ and $\mathrm{ker}E\neq 0$; anyway, note that $E$ is not surjective in this case).
\end{proof}

Now we are ready to establish the claims in the third column of Table 1.18.  In all dimension combinations listed there, the nontrivial element $v\in\mathrm{ker}E\cong \mathbb{Z}_2$ is the indicated Whitehead product of $\iota_{n-1}$ with a generator of $\pi_{m-n+1}(S^{n-1})$ (compare 3.8).  Moreover it follows from Lemma 5.1 that $\partial\circ E\equiv 2\cdot \mathrm{id}_{\pi_{m-1}(S^{n-1})}$ (\textit{cf.} 2.2).  In view of the discussion above and of Corollary 1.10 our claims for $q\neq 2$ are reduced to the question whether $v$ can be halved.  When $q=1$ this is just the "strong Kervaire invariant one problem" discussed e.g. in Theorem 1.13 and Section 4 above.  The following proofs contain detailed answers also for $q=2$ or 3.

\begin{flushleft}\textit{Proof of Theorem 1.16.}   We need to consider only the case $n\geqslant 4$, $n$ even (since $\pi_{m-1}(S^{n-1})$, $\partial$, $E\circ\partial$, $MCC(f,f)$ and $N^{\#}(f,f)$ all vanish for $n=2$).  We have the exact sequence\end{flushleft}
\begin{flushleft}
\begin{tikzpicture}
\matrix(m)[matrix of math nodes,
row sep=2.6em, column sep=2.8em,
text height=1.5ex, text depth=0.25ex]
{(5.2)\hspace{.57in} & \mathrm{ker}E\subset \pi_{2n-2}(S^{n-1}) & \pi_{2n-1}(S^n) & \mathbb{Z}\\};
\path[->,font=\scriptsize,>=angle 90]
(m-1-2) edge node[above] {$E$} (m-1-3)
(m-1-3) edge node[above] {$H\not\equiv 0$} (m-1-4);
\path[dashed,->,font=\scriptsize,>=angle 90]
(m-1-3) edge [bend left=35] node[below] {$\partial$} (m-1-2);
\end{tikzpicture}
\end{flushleft}\setcounter{equation}{2}
where the kernel of $E$ is generated by $v_{n-1}:=[\iota_{n-1},\eta_{n-1}]$ (\textit{cf.} 3.8).  Clearly $\mathrm{ker}H = E(\pi_{2n-2}(S^{n-1}))$ is just the torsion subgroup of $\pi_{2n-1}(S^n)$.  Moreover
\begin{equation}
\partial([\iota_n,\iota_n]) = v_{n-1}
\end{equation}
(\textit{cf.} [J], (7.4))\\
\indent If $n\neq 4, 8$, we can write 
$$
[\widetilde{f}] = E\alpha + \Big(\frac{1}{2}H([\widetilde{f}])\Big)\cdot [\iota_n,\iota_n];
$$
then
\begin{equation}
\partial([\widetilde{f}]) = 2\alpha + \Big(\frac{1}{2}H([\widetilde{f}])\Big)v_{n-1}
\end{equation}
(\textit{cf.} 2.2 and 5.1); hence $E\circ\partial([\widetilde{f}]) = 2E(\alpha)$.  This implies the first claim in Theorem 1.16 (\textit{cf.} 1.9).

\indent If $n\equiv 0$ (4), $E$ is injective (\textit{cf.} 1.18); hence the Wecken condition holds.

\indent Thus assume $n\equiv 2$ (4), $n\geqslant 6$.  Then $\mathrm{ker}E\cong \mathbb{Z}_2$ (\textit{cf.} 1.18) and $v_{n-1}$ is a nontrivial element of $\mathrm{ker}E\cap \partial(\pi_{2n-1}(S^n))$ (\textit{cf.} 5.3).  Therefore the Wecken condition fails (\textit{cf.} 1.8).  Moreover $v_{n-1}$ cannot be halved, \textit{i.e.} $v_{n-1}\not\in 2\pi_{2n-2}(S^{n-1})$ (\textit{cf.} [J], 5.2).  Thus $\partial([\widetilde{f}])=0$ if and only if $2\alpha = 0$ \underline{and} $\frac{1}{2}H([\widetilde{f}])$ is even (\textit{cf.} 5.4); in turn, $2\alpha=0$ precisely if $2E(\alpha)=E\circ\partial ([\widetilde{f}])$ or, equivalently, $N^{\#}(f,f)$ vanishes.  In view of Theorem 1.9 this completes the proof.\hspace{5.955in}$\Box$

For early examples illustrating Theorem 1.16 see [GR1], Section 4.  Related topics appear in [GR3].

Let us note one interesting aspect of the previous discussion.  The question whether the generator $v$ of $\mathrm{ker}E$ can be halved (\textit{i.e.} lies in $2\pi_{m-1}(S^{n-1}))$ plays crucial but different roles in the first two nonstable dimension settings.  If $q=1$ the Wecken condition fails and an additional looseness obstruction (the Kervaire invariant) is needed precisely when $v$ is both nontrivial and halvable.  In contrast, if $q=2$ the Wecken condition fails when only $v\neq 0$; then the Hopf invariant, taken mod 4, appears as the extra looseness obstruction (in addition to $N^{\#}$) because $v$ can \underline{not} be halved.

\begin{flushleft}
\textit{Proof of Theorem 1.17.}  Now assume $q=3$.  Then such halvability considerations play no role.  Indeed, $v=[\iota_{n-1},\eta^2_{n-1}]$ is even divisible by 4.  E.g. if $n\equiv 2$ (4), $n\geqslant 10$, then $v=4\alpha$ for a certain element $\alpha = \mathrm{J}(i_{n-1}(\mathbb{R})\overline{\tau}'_{n-2})$ of $\pi_{2n-1}(S^{n-1})$ (\textit{cf.} [GM], lines 10-13 of the introduction, together with formula (4.5)).  Theorem 1.17 follows now from our results 1.9, 1.10, 1.19, 5.1 and the injectivity and surjectivity criteria in Table 1.18 \hspace{1.3in}$\Box$\end{flushleft}

\begin{flushleft}
\textbf{Example 5.5: n = 2 or n = 6.}  Here the pair $(f,f)$ is loose by small deformation for all maps $f: S^{2n}\rightarrow S^n/\mathbb{Z}_2$.  Indeed, $\pi_{2n}(S^n) = \mathbb{Z}_2$ and hence $2[\widetilde{f}]=0$.  If $n=2$ and $\beta$ denotes the homotopy class of the Hopf map, then $[\iota_n,\iota_n]=\beta + \beta$, and $h_0([\widetilde{f}])\in\pi_4(S^3) =$ $E(\pi_3(S^2))$; thus $[\iota_n,\iota_n]\circ h_0(\widetilde{f}) = 2(\beta\circ h_0(\widetilde{f}))$ vanishes also. (As an alternative we may use Theorem 1.9 and the fact that $\partial ([\widetilde{f}])$ lies in $\pi_3(S^1)=0$.)\end{flushleft}

If $n=6$, then $[\iota_n,\iota_n]\circ h_0(\widetilde{f})=0$ by the following more general argument.  Given $n\geqslant 3$, the diagram
$$
\begin{tikzpicture}
\matrix(m)[matrix of math nodes,
row sep=.6em, column sep=2.4em,
text height=1.5ex, text depth=0.25ex]
{\mathbb{Z}_2\cdot\eta_n = \pi_{n+1}(S^n)\hspace{.2in} & \\
 & \pi_{2n}(S^n)\\
\phantom{\mathbb{Z}_2\cdot\eta =  }\pi_{2n}(S^{2n-1})\hspace{.2in} & \\};
\path[->,font=\scriptsize,>=angle 90]
(m-1-1) edge node[above] {$[\iota_n, - ]$}(m-2-2)
(m-1-1) edge node[left] {$\cong$} node[right] {$E^{n-1}$} (m-3-1)
(m-3-1) edge node[below] {$[\iota_n,\iota_n]_*$}(m-2-2);
\end{tikzpicture}
$$
commutes (\textit{cf.} [W], ch. XII, (2.4) and (2.5)).  Thus the induced homomorphism $[\iota_n,\iota_n]_*$ is injective precisely when $[\iota_n,\eta_n]\neq 0$, \textit{i.e.} $n\neq 6$ and $n\not\equiv 3$ $(4)$ (\textit{cf.} 3.3).  In particular, given any element $[\widetilde{f}]\in\pi_{2n}(S^n)$, we obtain also
\setcounter{prop}{5}
\begin{cor}
When $n\equiv 0$ $(4)$ and $2[\widetilde{f}]=0$ (e.g. when $n\leqslant 16$, \textit{cf.} [T]) then $(\widetilde{f},\widetilde{f})$ is loose if and only if $h_0([\widetilde{f}])\in\pi_{2n}(S^{2n-1})$ vanishes.
\end{cor}

\begin{flushleft}\textbf{Remark 5.7.}  It is important to notice that the Hopf homomorphisms $h_0$ appearing in 1.10, 1.17 and 5.6 on the one hand, and in 2.2 and 5.1 on the other hand, are defined in different dimension settings.  Since the parity of $j$ plays a big role in formulas (3.2) - (3.7) $h_0(\pi_{m-1}(S^{n-1}))$ may often vanish (e.g. as in 5.1) while $h_0(\pi_m(S^n))$ need not be trivial.  E.g. the Hopf invariant $h_0$ occurring in Corollary 5.6 above is never identically zero: $h_0(\pi_{2n}(S^n))\neq 0$ whenever $n\equiv 0$ (4) (since $E$ is not onto in this case).
\end{flushleft}

\end{document}